\documentclass[12pt,dvips]{amsart}
\usepackage{amsfonts, amssymb, latexsym, epsfig}
\usepackage{euler, amsfonts, amssymb, latexsym, epsfig, epic}
\newtheorem{Theorem}{Theorem}[section]

\newtheorem{Lemma}[Theorem]{Lemma}

\newtheorem{Corollary}[Theorem]{Corollary}

\newtheorem{Definition-Proposition}[Theorem]{Definition-Theorem}
\newtheorem{Main Conjecture}[Theorem]{Main Conjecture}

\theoremstyle{remark}
\newtheorem{Example}[Theorem]{Example}

\newcommand\LIS{{\tt LIS}}

\newcommand\tLambda{\widetilde\Lambda}
\newcommand\trLambda{\Lambda^\dagger}
\newcommand\trlambda{\lambda^\dagger}
\newcommand\tlambda{\widetilde\lambda}
\newcommand\tT{\widetilde T}
\newcommand\trT{T^\dagger}

\newcommand\trnu{\nu^\dagger}
\newcommand\tnu{\widetilde\nu}
\newcommand\tmu{\widetilde\mu}

\theoremstyle{plain}

\newcommand{\cellsize}{15}
\newlength{\cellsz} 
\newsavebox{\cell}
\sbox{\cell}{\begin{picture}(\cellsize,\cellsize)
\put(0,0){\line(1,0){\cellsize}}
\put(0,0){\line(0,1){\cellsize}}
\put(\cellsize,0){\line(0,1){\cellsize}}
\put(0,\cellsize){\line(1,0){\cellsize}}
\end{picture}}
\newcommand\cellify[1]{\def\thearg{#1}\def\nothing{}%
\ifx\thearg\nothing
\vrule width0pt height\cellsz depth0pt\else
\hbox to 0pt{\usebox{\cell} \hss}\fi%
\vbox to \cellsz{
\vss
\hbox to \cellsz{\hss$#1$\hss}
\vss}}
\newcommand\tableau[1]{\vtop{\let\\\cr
\baselineskip -16000pt \lineskiplimit 16000pt \lineskip 0pt
\ialign{&\cellify{##}\cr#1\crcr}}}
\newcommand{\kellsize}{30}
\newlength{\kellsz} 
\newsavebox{\kell}
\sbox{\kell}{\begin{picture}(\kellsize,\kellsize)
\put(0,0){\line(1,0){\kellsize}}
\put(0,0){\line(0,1){\kellsize}}
\put(\kellsize,0){\line(0,1){\kellsize}}
\put(0,\kellsize){\line(1,0){\kellsize}}
\end{picture}}
\newcommand\kellify[1]{\def\thearg{#1}\def\nothing{}%
\ifx\thearg\nothing
\vrule width0pt height\kellsz depth0pt\else
\hbox to 0pt{\usebox{\kell} \hss}\fi%
\vbox to \kellsz{
\vss
\hbox to \kellsz{\hss$#1$\hss}
\vss}}
\newcommand\ktableau[1]{\vtop{\let\\\cr
\baselineskip -16000pt \lineskiplimit 16000pt \lineskip 0pt
\ialign{&\kellify{##}\cr#1\crcr}}}

\newcommand{\sellsize}{36}
\newlength{\sellsz} 
\newsavebox{\sell}
\sbox{\sell}{\begin{picture}(\sellsize,20)
\put(0,0){\line(1,0){\sellsize}}
\put(0,0){\line(0,1){\sellsize}}
\put(\sellsize,0){\line(0,1){\sellsize}}
\put(0,\sellsize){\line(1,0){\sellsize}}
\end{picture}}
\newcommand\sellify[1]{\def\thearg{#1}\def\nothing{}%
\ifx\thearg\nothing
\vrule width0pt height\sellsz depth0pt\else
\hbox to 0pt{\usebox{\sell} \hss}\fi%
\vbox to \sellsz{
\vss
\hbox to \sellsz{\hss$#1$\hss}
\vss}}
\newcommand\stableau[1]{\vtop{\let\\\cr
\baselineskip -16000pt \lineskiplimit 16000pt \lineskip 0pt
\ialign{&\sellify{##}\cr#1\crcr}}}

\newcommand{\smellsize}{7}
\newlength{\smellsz} 
\newsavebox{\smell}
\sbox{\smell}{\begin{picture}(\smellsize,\smellsize)
\put(0,0){\line(1,0){\smellsize}}
\put(0,0){\line(0,1){\smellsize}}
\put(\smellsize,0){\line(0,1){\smellsize}}
\put(0,\smellsize){\line(1,0){\smellsize}}
\end{picture}}
\newcommand\smellify[1]{\def\thearg{#1}\def\nothing{}%
\ifx\thearg\nothing
\vrule width0pt height\smellsz depth0pt\else
\hbox to 0pt{\usebox{\smell} \hss}\fi%
\vbox to \smellsz{
\vss
\hbox to \smellsz{\hss$#1$\hss}
\vss}}
\newcommand\smtableau[1]{\vtop{\let\\\cr
\baselineskip -16000pt \lineskiplimit 16000pt \lineskip 0pt
\ialign{&\smellify{##}\cr#1\crcr}}}

\begin{document}
\pagestyle{plain}

\mbox{}
\title{$K$-theoretic Schubert calculus
for $OG(n,2n+1)$ and\\ Jeu de taquin for shifted increasing tableaux}

\author{Edward Clifford}
\address{Department of Mathematics, University of Maryland,
College Park, MD 20742, USA}
\email{ecliff@math.umd.edu}

\author{Hugh Thomas}
\address{Department of Mathematics and Statistics, University of New
Brunswick, Fredericton, New Brunswick, E3B 5A3, Canada}
\email{hugh@math.unb.ca}

\author{Alexander Yong}
\address{Department of Mathematics, University of Illinois at
Urbana-Champaign, Urbana, IL 61801, USA}
\email{ayong@math.uiuc.edu}
\date{August 14, 2010}

\begin{abstract}
We present a proof of a Littlewood-Richardson rule for the $K$-theory of
odd orthogonal Grassmannians $OG(n,2n+1)$, as conjectured in [Thomas-Yong '09].
Specifically, we prove that rectification using the {\it jeu de taquin} for increasing shifted tableaux
introduced there, is well-defined and gives rise to an
associative product. Recently, [Buch-Ravikumar '09] proved
a Pieri rule for $OG(n,2n+1)$ that
confirms a special
case of the conjecture. Together, these results imply the aforementioned
conjecture.
\end{abstract}
\maketitle

\vspace{-.2in}
\section{Introduction}
This paper concerns the {\it jeu de taquin} theory for shifted increasing tableaux
introduced in \cite{TY5}. Utilizing recent work of A.~Buch and V.~Ravikumar \cite{Bu}, we present a
proof of a rule, for the $K$-theoretic Schubert calculus of maximal isotropic odd orthogonal Grassmannians,
conjectured in \cite{TY5}.

\subsection{Main results}
Let $X=OG(n,2n+1)$ denote the Grassmannian of isotropic
$n$-dimensional planes in ${\mathbb C}^{2n+1}$, with respect to a symmetric, non-degenerate bilinear form.
The Schubert varieties $X_{\lambda}$ inside $X$ are indexed by partitions $\lambda$ with distinct parts and maximum part-size at most $n$.
Such a partition is identified with its
{\bf shifted Young diagram}, obtained from
the Young diagram for $\lambda$ by indenting the $i$-th row by $i-1$ columns.
Let $\Lambda$ denote the {\bf shifted staircase} associated to the partition $(n,n-1,\dots,1)$.
For example,
\begin{figure}[h]
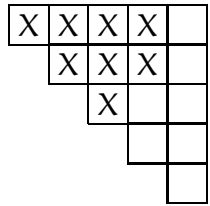

\[\tableau{{X}&{X }&{ X }&{ X }&{\ }\\&{X }&{X }&{ X }&{\ }\\&&{X }&{\ }&{\ }\\&&&{\ }&{\ }\\&&&&{\ }}\]
\caption{$\lambda=(4,3,1)\subset \Lambda$; $n=5$}
\end{figure}

The Grothendieck ring $K^{0}(X)$ of algebraic vector bundles over $X$ has a basis of Schubert structure
sheaves $[{\mathcal O}_{X_{\lambda}}]$ with which one defines the {\bf Schubert structure constants}:
\[[{\mathcal O}_{X_{\lambda}}] [{\mathcal O}_{X_{\mu}}]=\sum_{\nu}C_{\lambda,\mu}^{\nu}[{\mathcal O}_{X_{\nu}}].\]
Here the sum is over all shifted Young diagrams
$\nu$ with $|\nu|\geq |\lambda|+|\mu|$, where $|\lambda|=\lambda_1+\lambda_2+\cdots$.
When the equality
$|\nu|=|\lambda|+|\mu|$ holds, the numbers $C^\nu_{\lambda\mu}$ equal the classical Schubert structure constants for the cohomology
ring $H^{\star}(X,{\mathbb Q})$. The latter combinatorics is modeled by the multiplication of Schur $Q$-polynomials,
a class of symmetric functions. Combinatorial rules for the multiplication
of Schur $Q$-functions were found by
D.~Worley \cite{Worley} and J.~Stembridge \cite{Stembridge}; see also the textbook
by P.~Hoffman and J.~Humphreys \cite{Hoffman}. The connection to geometry was made by
H.~Hiller and B.~Boe \cite{Hiller.Boe} and P.~Pragacz \cite{Pragacz}. In view
of this, and the work on $K$-theoretic Littlewood-Richardson rules for ordinary Grassmannians (see, e.g.,
\cite{Buch:KLR, TY5} and the references therein),
it is a natural task to seek a rule for the more general $C_{\lambda,\mu}^{\nu}$. In \cite{TY5}, the first conjectural solution to this problem
was presented, and a proof strategy was outlined.

If $\lambda\subseteq \nu$, then the {\bf shifted skew shape} $\nu/\lambda$
consists of the boxes of
$\nu$ not in $\lambda$. A shape $\nu=\nu/\emptyset$ shall be referred to as {\bf straight}, in analogy with the
terminology for ordinary Young diagrams.
A {\bf shifted increasing tableau} of shape $\nu/\lambda$ is a filling of each of the boxes of
$\nu/\lambda$ by a nonnegative
integer such that each row and column is strictly increasing. In particular, a filling of a shape $\lambda$ is called
{\bf superstandard} if the first row consists of $1,2,\dots,\lambda_1$, the
second row consists of $\lambda_1+1,\dots,\lambda_1+\lambda_2$, etc. Denote this tableau by $S_\lambda$.
Let ${\tt INC}(\nu/\lambda)$ denote the set of all shifted increasing tableaux of shape $\nu/\lambda$.

\begin{figure}[h]
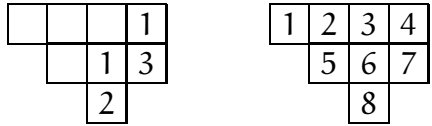

\[\tableau{{\ }&{\ }&{ \ }&{ 1 }\\&{\ }&{1 }&{3}\\&&{2 }} \ \ \ \ \ \ \ \ \ \
\tableau{{1 }&{2 }&{ 3 }&{ 4 }\\&{5 }&{6 }&{7}\\&&{8 }}
\]
\caption{A skew increasing tableau of shape $(4,3,1)/(3,1)$ and a superstandard tableau of shape
$(4,3,1)$}
\end{figure}
In \cite{TY5}, {\it jeu de taquin} and $K$-rectification
for shifted increasing fillings were introduced,
as part of a more general initiative towards $K$-theoretic Schubert calculus for minuscule $G/P$'s.
These notions extend ideas developed in \cite{TY2}, which in turn can be traced
back to the foundational work of
D.~Worley \cite{Worley}, M.-P.~Sch\"{u}tzenberger \cite{Schut}, and many others. This is reviewed in Section~2.

A $K$-rectification (or simply, rectification) of an increasing tableau $T$ is a tableau that results
from repeatedly applying {\it jeu de taquin} starting with $T$
until one reaches a straight shape tableau.
As in \cite{TY5}, in general, a tableau will not have a unique $K$-rectification.
However, we have the following main theorem, which was stated in
\cite[Section 7]{TY5}
as a conjectural
analogue
of~\cite[Theorem~1.2]{TY5}:
\begin{Theorem}\label{thm:main}
If $T$ is a shifted increasing tableau that rectifies to a superstandard
tableau $S_{\mu}$ using one $K$-rectification order, then $T$ rectifies to
$S_{\mu}$ using any $K$-rectification order.
\end{Theorem}

We can now state the following rule for the $K$-theoretic structure
constants, the conjecture from \cite{TY5} alluded to above.

\begin{Theorem}
\label{thm:main2}
$C_{\lambda\mu}^\nu$ equals $(-1)^{|\nu|-|\lambda|-|\mu|}$ times
the number of shifted increasing tableaux of
shape $\nu/\lambda$ that rectify to $S_\mu$.
\end{Theorem}

Among the  standard extensions of Schubert calculus (quantum, equivariant, $K$-theoretic and their
combinations), at present Theorem~\ref{thm:main2} represents the only complete Littlewood-Richardson type
rule for $OG(n,2n+1)$, either proved or even conjectural, beyond the setting of ordinary cohomology.

\begin{Example}
Let $\lambda=(3,1), \mu=(3,1), \nu=(5,3,1)$. The following six
increasing tableaux of shape
$(5,3,1)/(3,1)$ rectify to the superstandard tableaux $S_{(3,1)}$:
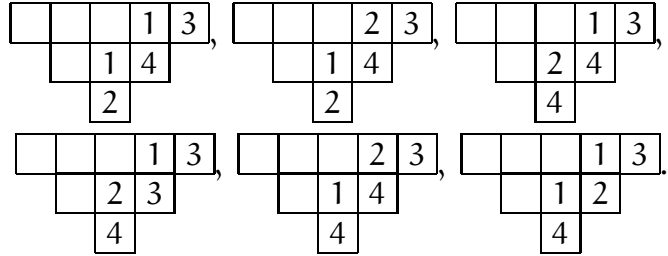
\begin{eqnarray*}
\tableau{\ & \ & \ & 1 & 3\\& \ & 1 & 4\\&& 2},
\
\tableau{\ & \ & \ & 2 & 3\\& \ & 1 & 4\\&& 2},
\
\tableau{\ & \ & \ & 1 & 3\\& \ & 2 & 4\\&& 4},\\
\tableau{\ & \ & \ & 1 & 3\\& \ & 2 & 3\\&& 4},
\
\tableau{\ & \ & \ & 2 & 3\\& \ & 1 & 4\\&& 4},
\
\tableau{\ & \ & \ & 1 & 3\\& \ & 1 & 2\\&& 4}.
\end{eqnarray*}
Hence $C_{(3,1),(3,1)}^{(5,3,1)}=(-1)^{9-4-4}(6)=-6$.\qed
\end{Example}

Together with the $K$-theoretic Littlewood-Richardson rule for Grassmannians
proved in \cite{TY5}, this work completes the proof of two key cases of the root-system uniform conjecture
for minuscule $G/P$'s  stated in that paper.
Note that since there is an isomorphism between $OG(n,2n+1)$ and (a component of) $OG(n+1,2n+2)$, Theorem~\ref{thm:main2}
applies to Schubert calculus of the latter space as well. Using the
(equivariant) $K$-theory Monk-Chevalley formula due to C.~Lenart--A.~Postnikov \cite{Lenart.Postnikov}, one can easily
 check the remaining minuscule classical type cases of the main conjecture of \cite{TY5} (i.e., projective spaces and even-dimensional quadrics); see \cite{Clifford}.

Using Theorem~\ref{thm:main2}, it should be possible to modify
the ``cominuscule recursion'' ideas set forth in \cite{TY2} to complete the proof of the
conjecture in the exceptional type cases ($E_6$ and $E_7$). (Alternatively, a computer aided proof is plausible.)
We also expect that one can combine the Pieri rule for maximal Lagrangian Grassmannians $LG(n,2n)$ \cite{Bu},
together with the ideas of this paper to prove a $K$-theory Littlewood-Richardson rule for that homogeneous space.
We may detail these ideas elsewhere.

There is a second ${\mathbb Z}$-linear basis of $K^{0}(X)$ given by
the classes of ideal sheaves of the boundary of Schubert varieties
$\partial X_\lambda=X_{\lambda}\setminus X_{\lambda}^{\circ}$ where
$X_{\lambda}^{\circ}$ is the open Schubert cell corresponding to $\lambda$, see, e.g.,
\cite[Section~2]{TY8} and the references therein. Define coefficients by
\[[\partial X_{\lambda}][\partial X_{\mu}]=\sum_{\nu}E_{\lambda,\mu}^{\nu}[\partial X_\nu].\]
Define a superset ${\widehat {\tt INC}(\nu/\lambda)}$ of
the set of increasing tableaux, by filling $\nu/\lambda$
with positive integers and the label ``$X$'' on any number of the outer corners of $\nu/\lambda$. We still demand
the rows and columns to be strictly increasing; the labels $X$
are assumed to have value $\infty$ insofar as this condition is concerned. We can still speak of $K$-rectification
of such tableaux, by first deleting any $X$'s before beginning the process.

In \cite{TY8}, the Grassmannian analogue of the following result was proved:

\begin{Corollary}
\label{cor:main3}
The structure constant $E_{\lambda,\mu}^{\nu}$ equals $(-1)^{|\nu|-|\lambda|-|\mu|}$ times the number of tableaux $T\in {\widehat {\tt INC}}(\nu/\lambda)$ that rectify to $S_{\mu}$.
\end{Corollary}

\begin{Example}
Let $\lambda=(2),\mu=(1)$ and $\nu=(3,1)$. Then the following three tableaux in
${\widehat {\tt INC}((3,1)/(2))}$ witness $E_{(2),(1)}^{(3,1)}=-3$:
$\tableau{\ & \ & 1\\&X}, \ \tableau{\ & \ & X\\&1}, \ \tableau{\ & \ & 1\\&1}$.\qed
\end{Example}

\subsection{Organization of the proof}
The proof follows the strategy outlined in \cite{TY5}, and uses an
analogue of \cite[Lemma~5.1]{TY5},  Lemma~\ref{lemma:crucial} below, which asserts that any rule computing
numbers $\{D_{\lambda,\mu}^{\nu}\}$ giving rise to an associative product on a ring with basis $\{[\lambda]\}$, and
which satisfy $C_{\lambda,(t)}^{\nu}=D_{\lambda,(t)}^{\nu}$ for all shapes $\lambda,\nu$ and
for all $1\leq t\leq n$,
must actually satisfy
$C_{\lambda,\mu}^{\nu}=D_{\lambda,\mu}^{\nu}$ for all $\lambda,\mu,\nu$.

We will apply this lemma to the numbers
$D_{\lambda\mu}^\nu$ given by our {\it jeu de taquin} rule.
A.~Buch and V.~Ravikumar \cite{Bu} recently
established a Pieri rule for computing $C_{\lambda,(t)}^{\nu}$.  We show that
the numbers produced by our rule in the Pieri case agree with their rule.
Hence, this paper completes our aforementioned strategy by giving a
proof of Theorem \ref{thm:main} and using the ideas of \cite{TY5} to show it
implies the desired associativity claim.
It is reported in \cite{Bu} that
I.~Feigenbaum and E.~Sergel also have a proof (unpublished) of the agreement between
our rule in the Pieri case and the Pieri rule of \cite{Bu}.

In Section~2 we review the definition of
{\it jeu de taquin} for
increasing tableaux as announced in \cite{TY5} and discuss properties of it
that we will need. In Sections~3 and 4 we prove Theorem~\ref{thm:main},
formulate the aforementioned Lemma~\ref{lemma:crucial},
and then use \cite{Bu} and Theorem~\ref{thm:REU} to complete the proof of Theorem~\ref{thm:main2}.
The key idea in our proof of
Theorem~\ref{thm:main}
is to use results proved in \cite{TY5}, together with
a ``doubling'' argument that rephrases the shifted problem as a non-shifted problem.  In Section~5, we show how to modify our arguments from \cite{TY8} to
prove Corollary~\ref{cor:main3}.

\section{Jeu de taquin for increasing tableaux}

In \cite{TY5}, a notion of {\it jeu de taquin} was introduced,
that applies to shapes
associated to a certain subposet $\Lambda$ of the poset of positive roots, and a minuscule
simple root. In the case of type $B_n$, $\Lambda$ is the shifted staircase and the shapes are precisely
the shifted shapes from Section~1. We now review this case.

Given $T\in {\tt INC}(\nu/\lambda)$, we say that an {\bf inner corner} is a maximally southeast
box $x\in\lambda$. More generally we are interested in a collection of inner corners
${\bf x}=\{x_i\}$. A {\bf short ribbon} ${\mathfrak R}$ is an edge-connected shifted skew-shape with at most two boxes in any row
or column and no $2\times 2$ subshape. $\mathfrak R$ is an {\bf alternating ribbon} if it is filled with two symbols where
adjacent boxes are filled differently. For such a ribbon, define
${\tt switch}({\mathfrak R})$ to have the same shape as
${\mathfrak R}$, and the filling obtained by interchanging the two symbols,
except that  if ${\mathfrak R}$ consists of a single box,
$\tt switch$ does nothing to
it.
Define $\tt switch$ to act on a union of alternating ribbons,
by acting on each separately.
For example:
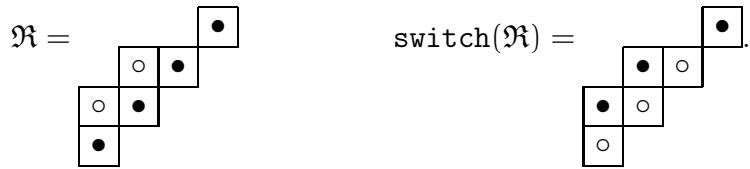
\[{\mathfrak R}=\tableau{&&&{\bullet }\\&{\circ}&{\bullet}\\{\circ}&{\bullet}\\{\bullet}}
\mbox{ \ \ \ \ \ \ \ \ \ \ \ \ \ \
${\tt switch}({\mathfrak R})=
\tableau{&&&{\bullet}\\&{\bullet}&{\circ}\\{\bullet}&{\circ}\\{\circ}}$.}
\]

Fix inner corners ${\bf x}=\{x_1,\ldots,x_s\}$, and fill
each with a ``$\bullet$''.  Let ${\mathfrak R}_1$ be the
union of alternating ribbons made of boxes using $\bullet$ or $1$.  Apply
$\tt switch$ to ${\mathfrak R}_1$.  Now let ${\mathfrak R}_2$ be the union of alternating ribbons
using $\bullet$ or $2$, and proceed as before. Repeat until the $\bullet$'s
have been switched past all the entries of $T$.
The final placement of the numerical
entries gives  $K{\tt jdt}_{\{x_i\}}(T)$. It is easy to check that
if $T$ is an increasing tableau, then so is $K{\tt jdt}_{\{x_i\}}(T)$.

\begin{Example}
If ${\bf x}$ consists of the boxes with the $\bullet$'s in the depiction of $T$, then
we have:
\[T=
\tableau{{\ }&{\ }&{\bullet}&{1}\\&{\bullet}&{1}&{3}\\&&{2}}
\mapsto
\tableau{{\ }&{\ }&{1}&{\bullet}\\&{1}&{\bullet}&{3}\\&&{2}}
\mapsto
\tableau{{\ }&{\ }&{1}&{\bullet}\\&{1}&{2}&{3}\\&&{\bullet }}
\mapsto
\tableau{{\ }&{\ }&{1}&{3}\\&{1}&{2}&{\bullet}\\&&{\bullet}}
=K{\tt jdt}_{{\bf x}}(T)
\]\qed
\end{Example}

A sequence of $K{\tt jdt}$ operations applied successively to $T\in {\tt INC}(\nu/\lambda)$
that lead to a straight shape tableau $U\in {\tt INC}(\mu)$ is called a {\bf $K$-rectification},
and the said sequence is referred to as a {\bf $K$-rectification order}. Note that a $K$-rectification order
is naturally encoded by a shifted increasing tableau $R\in {\tt INC}(\lambda)$ whose largest labels indicate the
corners of $\lambda$ used in the first $K{\tt jdt}$-slide in the $K$-rectification, and whose second largest labels indicate
the corners used in the second $K{\tt jdt}$-slide, etc.

The same definitions were given and studied for increasing tableaux of non-shifted shapes in \cite{TY5}.
With the above definitions,
one has \emph{mutatis mutandis} extensions of Sections~2 and~3 of \cite{TY5}.
In particular, given
$T\in {\tt INC}(\lambda/\alpha) \mbox{\ and $U\in {\tt INC}(\nu/\lambda)$}$
we can define
\[K{\tt infusion}(T,U)=(K{\tt infusion}_1(T,U), K{\tt infusion}_2(T,U))\in {\tt INC}(\gamma/\alpha)\times
{\tt INC}(\nu/\gamma)\]
(for some straight shape $\gamma$). For brevity we do not repeat these definitions here.

Although our main purpose in this paper is to discuss
the shifted case, our arguments run, in part, through the non-shifted setting. We
also refer the reader to \cite{TY5} for definitions and details.

In particular, in \cite{TY5},
the {\bf longest (strictly) increasing subsequence}
of a non-shifted increasing tableau $T$ is defined.
This is a collection
of boxes of $T$, of maximal size, such that reading the labels of the boxes along
rows, from left to right, and from bottom to top, gives a strictly increasing sequence of numbers.  We denote the length of a
longest increasing
subsequence by ${\tt LIS}(T)$. In \cite{TY5}, the following result was
obtained:

\begin{Theorem}[{\cite[Theorem 6.1]{TY5}}]\label{thm:sixone}
 If $T$ is a non-shifted increasing tableau, then
the length of the first row of any (non-shifted) $K$-rectification of $T$ is
$\LIS(T)$.
\end{Theorem}

\section{The doubling argument}
Fix the rectangular shape
\[\tLambda=(n+1)\times n.\]
Observe that
the boxes of $\tLambda$ can be divided into two parts, a copy of $\Lambda$ in the top right,
and a copy of the transpose of $\Lambda$ in the bottom left, that we denote $\trLambda$. For $\lambda$ a shape in $\Lambda$,
define $\trlambda$ to be the transpose of $\lambda$ in placed into
$\trLambda$. Define
\[\tlambda=\lambda \cup \trlambda.\]
Note that this is a non-shifted shape in $\tLambda$.

Let $T$ be an increasing filling of $\nu/\lambda$ in $\Lambda$.
We will define a filling
$\tT$ of $\tnu/\tlambda$, as follows.
Fill the copy of $\trnu/\trlambda$ inside $\tLambda$ with
entries obtained by doubling the corresponding entries of $T$.  Fill
the copy of $\nu/\lambda$ inside $\Lambda$ with entries obtained by doubling
the entries of $T$ and subtracting 1.
Consequently, $\tT\cap \Lambda$ is a tableau with
odd number entries whose
{\bf flattening} is $T$, i.e., the tableau produced by replacing the labels
that appear in $\tT\cap\Lambda$ by $\{1,2,3\dots\}$ in an order-preserving
fashion, is $T$. Meanwhile
$\tT\cap \trLambda$ is a tableau with even number entries whose
flattening is the transpose of $T$. It is easy to see that $\tT$ is increasing.

\begin{Example}
For instance, the skew increasing tableaux given in Section~2 is doubled as follows:
\[\begin{picture}(160,110)
\put(-5,55){$T=\ $\tableau{{\ }&{\ }&{ \ }&{ 1 }\\&{\ }&{1 }&{3}\\&&{2 }}}
\put(75,50){$\ \ \mapsto$}
\put(100,75){\tableau{{}&{}&{}&{1}&{\ }\\
{}&{}&{1}&{5}&{\ }\\
{}&{}&{3}&{\ }&{\ }\\
{}&{2}&{4 }&{\ }&{\ }\\
{2}&{6}&{\ }&{\ }&{\ }\\
{\ }&{\ }&{\ }&{\ }&{\ }}}
\put(180,55){$={\widetilde T}$}
\linethickness{2pt}
\put(100,75){\line(1,0){16}}
\put(115,75){\line(0,-1){16}}
\put(115,60){\line(1,0){16}}
\put(130,60){\line(0,-1){16}}
\put(130,45){\line(1,0){16}}
\put(145,45){\line(0,-1){16}}
\put(145,30){\line(1,0){16}}
\put(160,30){\line(0,-1){16}}
\put(160,15){\line(1,0){16}}
\end{picture}
\]
\qed
\end{Example}

The main result of this section is the following lemma:

\begin{Lemma} \label{lem:double}
If $T$ in $\Lambda$ rectifies to $V$ with respect to
the $K$-rectification order encoded by the shifted increasing tableau $S$, then
$\tT$ rectifies to $\widetilde V$ with
respect to the $K$-rectification order encoded by the non-shifted increasing tableau $\widetilde S$.
\end{Lemma}

\begin{proof}
Let ${\bf x}$ be an collection of inner corners of $T$, and ${\bf x}^\dagger$ the corresponding collection of
inner corners of $\trT$.  It suffices to show that
$$K{\tt jdt}_{{\bf x}}K{\tt jdt}_{{\bf x}^\dagger} \widetilde T= \widetilde{K{\tt jdt}_{{\bf x}} T}.$$
(Here the lefthand side is a pair of non-shifted $K{\tt jdt}$-slides, while the righthand side is a single shifted $K{\tt jdt}$-slide.) To show this, it is sufficient to show that $K{\tt jdt}_{{\bf x}^\dagger}$ applied to
$\widetilde T$ will only alter $\widetilde T \cap \trLambda$, and then
that the subsequent application of $K{\tt jdt}_{\bf x}$ will only alter $\widetilde T
\cap \Lambda$. Given this, it is clear that both operations mimic how $K{\tt jdt}_{\bf x}$ acts on
$T$ itself (up to a relabeling).

When considering $K{\tt jdt}_{{\bf x}^\dagger}$, the desired claim is clear
for all steps of the slide except when the current switch involves a $\bullet$ on the sub-diagonal
(that is to say, on the northwest-southeast diagonal boxes of $\trLambda$).  We wish
to know that it is the entry below the $\bullet$ which will move into it, rather
than the entry to its right (which is in $\Lambda$).

It is not difficult to argue that in such a situation the local picture of
the slide is as follows: the entries $a,b,c$ are in $\Lambda$, with corresponding
entries $\bullet,b^\dagger,c^\dagger$ in $\trLambda$. In particular, $b^\dagger$ and
$c^\dagger$ have not yet moved from their pre-slide position. Hence, $b^\dagger=b+1$ and $c^\dagger=c+1$.

\[\begin{picture}(160,40)
\put(60,30){$\ldots \ \tableau{  a&b   \\
\bullet & c\\
            b^\dagger&c^\dagger} \ldots$}
\linethickness{2pt}
\put(80,45){\line(-1,0){16}}
\put(80,30){\line(0,1){16}}
\put(80,30){\line(1,0){16}}
\put(95,30){\line(0,-1){16}}
\put(95,15){\line(1,0){16}}
\put(110,15){\line(0,-1){16}}
\put(110,0){\line(1,0){16}}
\end{picture}\]

Since $b<c$, we know $b^\dagger<c$, so when we apply a {\tt switch} which
replaces the $\bullet$ with a numerical entry, it is $b^\dagger$
which will switch with the $\bullet$ (and possibly a $\bullet$ to its left, although this does not affect
our argument), leaving the entries
in $\Lambda$ unaltered. Hence at the completion of $K{\tt jdt}_{{\bf x}^{\dagger}}({\widetilde T})$,
one obtains exactly $K{\tt jdt}_{{\bf x}}(T)$, after transposing and halving the values of each label.
The part of ${\widetilde T}$ inside the copy of ${\Lambda}$ is unchanged.

Now, let $W=K{\tt jdt}_{{\bf x}^{\dagger}}({\widetilde T})$.  We now consider the
effect of applying $K{\tt jdt}_{\bf x}$ to $W$ (inside $\widetilde \Lambda$)
and to $T$ inside $\Lambda$.
Again, the two calculations clearly proceed in exactly the same way, except
perhaps when a $\bullet$ lies on the diagonal, so we focus on that situation.
In the computation of $K{\tt jdt}_{\bf x}(T)$, at the step in which the $\bullet$ is replaced
by a numerical entry, it must be the label to the right of the $\bullet$
which is switched into the diagonal box; we must confirm that the same
thing happens in the present calculation in $\widetilde \Lambda$.

Suppose that $a^{\dagger}$ is the label just below the $\bullet$ in $\widetilde
\Lambda$.
By the computation of
$W=K{\tt jdt}_{{\bf x}^{\dagger}}({\widetilde T})$ which we just did, we know that
$W\cap \trLambda$ agrees with $K{\tt jdt}_{\bf x}(T)$ (up to the doubling of the
entries).
Then, since (by induction) all the steps of
our present computation of $K{\tt jdt}_{\bf x}(W)$ precisely mimic the moves of
$K{\tt jdt}_{\bf x}(T)$ up until this diagonal move,
we must have $a$ to the right of the
$\bullet$ in $\Lambda$, with $a^\dagger=a+1$. Thus, we find ourselves in the following
situation:
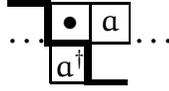
\[\begin{picture}(160,35)
\put(60,20){$\ldots \tableau{\bullet &a\\
          a^\dagger}\ldots$}
\linethickness{2pt}
\put(60,35){\line(1,0){16}}
\put(75,36){\line(0,-1){16}}
\put(74,20){\line(1,0){16}}
\put(90,21){\line(0,-1){16}}
\put(89,5){\line(1,0){16}}
\end{picture}
\]
Since $a^\dagger>a$, $a$ will move to the right in our calculation of
$K{\tt jdt}_{\bf x}(W)$, as desired.
\end{proof}

\section{Proofs of the main results}

A {\bf vertical strip} is a collection of boxes, no two in the same row, such that each box is weakly west
of the boxes above it.  Similarly, a {\bf horizontal strip} is a collection of boxes, no two in the same column, such that each box is weakly north of the boxes which are west of it.
Define a {\bf $t$-Pieri filling} to be an increasing filling $T$
of a skew shape in $\Lambda$
using
 all of the numbers
$1,\dots, t$
with the property that $T$ consists of:
\begin{itemize}
 \item[(i)] a vertical strip with entries increasing from northeast to southwest,
with repeated entries
allowed but necessarily in different columns, and
\item[(ii)] a horizontal strip with entries increasing from southwest to northeast,
with repeated entries
allowed but necessarily in different rows,
\item[(iii)] such that all the entries in the vertical strip are
less than or equal to the entries in the horizontal strip.
\end{itemize}

We say that a filling which includes $\bullet$'s is a $t$-Pieri filling
if its numbered boxes can be divided into a horizontal and vertical strip satisfying the above conditions
and such that the $\bullet$'s are positioned southwest to northeast, with no two in any given row or column.

\begin{Example}
The following are examples of $4$-Pieri fillings, without and with $\bullet$'s:

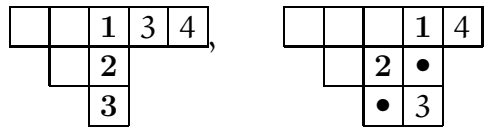
$$\tableau{{\ }&{\ }&{{\bf 1}}&{3}&{4}\\&{\ }&{{\bf 2}}\\&&{{\bf 3}}},
\qquad \tableau{{\ }&{\ }&{\ }&{\bf 1}&4\\&{\ }&{\bf 2}&{\bullet}\\& & {\bullet}&3}$$
We have used boldface to highlight the vertical strip.\qed
\end{Example}

A shifted skew shape containing no $2\times 2$ square is called a
{\bf ribbon}. Note that a $t$-Pieri filling without $\bullet$'s necessarily has ribbon shape.
Such fillings appear at the beginning and end of the
induction argument of Lemma~\ref{lem:goodrect} below.

\begin{Lemma} \label{Cliffa} If $T$ is a $t$-Pieri filling (possibly containing
$\bullet$'s) then there is some number $k$ (not necessarily unique) such that the boxes with labels
from $1$ to $k$ form a vertical strip, and the boxes with labels from
$k+1$ to $t$ form a horizontal strip. \end{Lemma}

\begin{proof} By assumption, $T$ can be viewed as the disjoint union of
a vertical strip of boxes, $V$, and a horizontal strip of boxes, $H$, with
the labels of $V$ weakly less than the labels of $H$, together with a set of
boxes filled with $\bullet$'s.  If there is no
overlap between the labels of $V$ and the labels of $H$, then we are
done, i.e., we can set $k=K$ where $K$ is the maximum value of a label in $V$.
Otherwise, suppose this $K$ occurs in both $V$ and $H$.
We show that there exists a (possibly) different way to subdivide $T$ into
a vertical strip $V'$ and a horizontal strip $H'$, as in the definition of
$t$-Pieri filling, and such that $V'$ and $H'$ have no label
in common.

If we could take $H'$ to consist of all the boxes numbered $K$ to $t$
(and would thereby obtain a subdivision as in the definition of $t$-Pieri
filling), we would be done with $k=K-1$.  So suppose otherwise.
This implies that there is some box labeled $K$ which is
north of a box labeled $K+1$.  This box labeled $K$ must be in $V$.  In
particular, this implies that the most northerly box labeled $K$
is in $V$, and therefore that we may take $V'$ to consist of all boxes
labeled $1$ to $K$, $H'$ to consist of the remaining boxes, and hence $k=K$.
\end{proof}

\begin{Lemma} \label{lem:goodrect} If $T$ is a $t$-Pieri filling (without $\bullet$'s) then it rectifies to
$S_{(t)}$ for any rectification order.
\end{Lemma}

\begin{proof}
Choose a rectification order for $T$.
Note that the only $t$-Pieri filling (without $\bullet$'s) of a straight shape is
$S_{(t)}$. Hence it  is sufficient to
show that each {\tt switch} which occurs in the rectification, preserves
the property of being $t$-Pieri.

Let $R$ be some tableau in the sequence of tableaux produced by successive
{\tt switch}es; by induction, we assume that $R$ is $t$-Pieri.  Let
$R'$ be obtained from $R$ by {\tt switch}ing the $\bullet$'s and the
$i$'s.  We wish to show that $R'$ is also $t$-Pieri.

By Lemma \ref{Cliffa},
we know that $R$ can be described as the union of a vertical strip with
entries $1$ to $k$ and a horizontal strip with entries $k+1$ to $t$, together
with some $\bullet$'s.
We
split into two cases, depending on whether or not $i\leq k$.

\noindent
{\bf Case I}: $i\leq k$.  The {\tt switch} affects only $V$.  The only problem
that could occur would be if there were a $\bullet$ to the left of an $i$, and
an $i+1$ (also in $V$) in the same column as the $i$ and below it.
It is straightforward to verify that since the shape of $R$ is skew,
this is only possible in the following configuration, where the box below the $\bullet$ is not in $\Lambda$ (more explicit details are available in the proof of \cite[Lemma~3.1.4]{Clifford}):
$$\ktableau{\bullet&i\\&i+1}.$$
In this case, no other box labeled $i+1$ can occur in $V$, and no box
with a larger label can occur in $V$.  Therefore $i+1=k$.  Define $V'$
to consist of the boxes of $R'$ labeled at most $k-1$, and define
$H'$ to consist of the boxes of $R'$ labeled at least $k$.  This decomposition
shows that $R'$ is $t$-Pieri.

\noindent{\bf Case II}: $i>k$.  The switch affects only $H$.  The only
problem that could be is if there were a $\bullet$ above $i$, and
an $i+1$ in the same row, while the box to the right of the $\bullet$ is not in $R$, as depicted below:
$$\ktableau{\bullet&\\
             i&i+1}.$$

However, this is impossible, since the shape of $R$ is skew.
\end{proof}

For $\nu/\lambda$ a skew shape in $\Lambda$, write
$(\nu/\lambda)^\vee$ for the shape obtained by reflecting
$\nu/\lambda$ in the main (southwest to northeast) antidiagonal in $\Lambda$.
For an increasing filling $R$ of $\Lambda$, define
$R^\vee$ to be the filling of shape $({\tt shape}(R))^\vee$
obtained by filling the positions corresponding to entries $i$ in $R$
with entries $\max R +1 -i$.  Note that this operation takes increasing fillings
to increasing fillings.

\begin{Lemma} \label{lemma:rotating}
If $T$ is an increasing filling then $T^\vee$ is an increasing filling,
Moreover if $T$ is $t$-Pieri then so is $T^\vee$.
\end{Lemma}

\begin{proof}
For the first assertion suppose $x$ and $y$ are two adjacent boxes in $T$
with $x$ to the left of $y$. Then the label of $x$ is strictly smaller than the
one for $y$. In $T^\vee$, boxes $x$ and $y$ correspond to boxes $x'$ and $y'$
where
$y'$ is north of $x'$, but by construction, the label of $y'$ is now strictly
smaller than the label of $x'$. A similar argument holds if $x$ is north of $y$. The first assertion then follows.

For the second claim, consider the horizontal strip $H'$ using labels $k,k+1,\ldots,t$
that comes from the fact that $T$ is $t$-Pieri. This strip increases from southwest to northeast. Under the reflection, it is sent to a vertical strip, and each
label $\ell$ is replaced by $t-\ell+1$. Hence in $T^\vee$ we have a
vertical strip on the labels $t-k+1,t-k,\ldots,3,2,1$ that decreases from southwest
to northeast. Similarly, the vertical strip in $T$ guaranteed by $T$ being $t$-Pieri
corresponds to a horizontal strip in $T^\vee$ on the labels $t-k+1,t-k+2,\ldots,t$. Thus it follows that $T^\vee$ is $t$-Pieri.
\end{proof}

\begin{Lemma} \label{lem:pieri} If $T$ is an increasing filling of a skew shape in $\Lambda$ which
rectifies to $S_{(t)}$ for some rectification
order, then $T$ must be a $t$-Pieri filling.
\end{Lemma}

\begin{proof}
Suppose otherwise, that there exists a filling $T$ that is not $t$-Pieri,
but rectifies to $S_{(t)}$.
Following the reverse of the rectification order that takes $T$
to $S_{(t)}$, clearly one can partially rectify
$(S_{(t)})^\vee$ to $T^\vee$.

By Lemma~\ref{lemma:rotating}
$(S_{(t)})^\vee$ is $t$-Pieri while
$T^\vee$ is not $t$-Pieri.
This violates the proof of
Lemma \ref{lem:goodrect} which says that any sequence of slides applied to
a $t$-Pieri filling is still $t$-Pieri.
\end{proof}

Together, these two lemmas imply the following theorem.

\begin{Theorem} \label{thm:REU}
If an increasing tableau $T$
is a $t$-Pieri filling then it rectifies to
$S_{(t)}$ for any $K$-rectification order.
Moreover, if $T$ is an increasing filling of a skew shape in $\Lambda$ that
rectifies to $S_{(t)}$ for some $K$-rectification
order, then $T$ must be a $t$-Pieri filling.
\end{Theorem}

In order to state the Pieri rule for $OG(n,2n+1)$ given in \cite[Corollary~4.8]{Bu}, we need
one further definition from their paper.
Suppose $\nu/\lambda$ is a ribbon. Then a {\bf KOG-tableau} of shape $\nu/\lambda$ is a tableau in ${\tt INC}(\nu/\lambda)$ with the additional constraint that the label of
any box $b$ is:
\begin{itemize}
\item[(A)] either smaller than or equal to all the boxes southwest of it, or
\item[(B)] greater than or equal to all the boxes southwest of it.
\end{itemize}
Using this definition, their Pieri rule states
$C_{\lambda,(t)}^{\nu}$ equals $(-1)^{|\nu/\lambda|-t}$ times the number of KOG-tableaux
with shape $\nu/\lambda$ and labels $1,2,\ldots,t$.

Thus, to connect the Buch-Ravikumar rule to our conjectural rule, it remains to prove
the following:

\begin{Lemma}\label{translation} The set of
$t$-Pieri tableaux (without $\bullet$'s) of shape $\nu/\lambda$ equals the set of {\rm KOG}-tableaux of
the same shape that use the labels $1,2,\ldots,t$.
\end{Lemma}
\begin{proof}
Suppose $T$ is a $t$-Pieri tableau for some value of $k$ as given in Lemma~\ref{Cliffa}. Now consider
any box $b$ of $T$. If it is on the vertical strip using the labels $k,k-1,\ldots,3,2,1$
from the definition of $T$ being $t$-Pieri, then (A) holds with respect to all
the other labels southwest of it and on that vertical strip. Also, any
label on the horizontal strip is of size $k$ or larger, so (A) clearly holds with
respect to them as well. On the other hand, suppose $b$ is on the horizontal strip
using labels $k,k+1,\ldots,t-1,t$.  Then (B) holds with respect to any other label
southwest of it and on that strip. Also, (B) clearly holds with respect to any label
on the vertical strip, since that other label is of size at most $k$. Hence $T$ is a
KOG-tableau.

Conversely, now suppose $T$ is a KOG-tableau with the stated assumptions. Initially
set $k=\infty$. Consider
the the southwest most label $1$ that appears in $T$. If all labels $2$ appear
weakly northeast of that $1$ then set $k=1$. Otherwise, look at the southwest most $2$ that appears. If all labels $3$ are weakly northeast of that $2$ set $k=2$.
Repeat this process until $k<\infty$ or until we have exhausted all labels, at which point set $k=t$.

We claim that $T$ is $t$-Pieri. Consider the labels $1,2,\ldots,k-1$. By the fact $T$ is a KOG-tableau, and by our construction of $k$, these labels form a vertical
strip that decreases from southwest to northeast. At least one $k$ appears southwest
of the southwest most $k-1$. Then the labels $1,2,\ldots,k-1$ together with all such
labels $k$ form the desired vertical strip from the definition of $t$-Pieri tableau.
In particular, the southwest most $k$ that appears in $T$ is part of the vertical
strip.

We begin a subtableau $H$ with that $k$ together with any $k$'s
that lie weakly northeast of the southwestmost $k-1$ we previously considered and
all labels $k+1,k+2,\ldots,t$ from $T$.
Note that the $k+1$'s must all lie weakly northeast of the northeastmost of these
$k$'s, by (B). Similarly, the $k+2$'s must all lie weakly northeast of the northeastmost
of the $k+1$'s etc. Hence $H$ is in fact a horizontal strip and $T$ is a $t$-Pieri tableau, as desired.
\end{proof}

In \cite{Bu} it is stated that I.~Feigenbaum--E.~Sergel have given a proof
(unpublished) of the following
corollary. We present a proof that we obtained independently:

\begin{Corollary}
\label{cor:REUclaim}
$C_{\lambda,\rho}^{\nu}$
is correctly computed by the rule of Theorem~\ref{thm:main2}
whenever $\rho=(t)$.
\end{Corollary}
\begin{proof}
This holds by Theorem~\ref{thm:REU}, Lemma~\ref{translation} and
\cite[Corollary~4.8]{Bu} combined.
\end{proof}

\subsection{Proof of Theorem~\ref{thm:main}}
Let $T$ be a tableau in $\Lambda$ that rectifies to a
superstandard tableau of shape $\mu$ with respect to some $K$-rectification
order $S$.

The proof is by induction on the size of $T$.

Consider some alternative $K$-rectification order $R$.
Since $T$ rectifies to superstandard with respect to $S$, it contains a
$\mu_1$-Pieri subtableau $A$, by the second statement of Theorem~\ref{thm:REU}. In particular, it is easy to see that we can
assume that $A$ contains all the boxes of $T$ that use the labels $1,2,\ldots,\mu_1$.
By the first statement of Theorem~\ref{thm:REU}, $A$ rectifies to the row
$(\mu_1)$ with respect to $R$ also.   Let $B=T\setminus A$,
and
$K{\tt infusion}(R,A)=(S_{(\mu_1)},\hat R)$.
By \cite[Lemma~3.3]{TY5} (adapted to the shifted setting, with the same proof),
since the labels used in $A$ and $B$ form consecutive, disjoint intervals,
one can split the computation of the $K$-rectification of $A$ with respect to $R$
and the remaining $K$-rectification of $B$ with respect to $\hat R$.

By Lemma \ref{lem:double} we know that $\tT$ rectifies to a tableau of shape $\tmu$ under the
$K$-rectification order ${\widetilde S}$.
By Theorem \ref{thm:sixone}, $K$-rectifying $\tT$ with respect to any order will
result in a shape whose first row has the same length as does $\tmu$, namely,
$\mu_1$.  Thus,
in particular, the further $K$-rectification of $B$ with respect to $\hat R$ will
not add any elements to $S_{(\mu_1)}$.  Therefore, we can forget about the
first row of $\Lambda$.
We know that $B$ rectifies to superstandard (on the labels starting with $\mu_1 +1$)
with respect to some order (since $T$ does).  By induction, $B$ rectifies to
superstandard with respect to $\hat R$, and we are done.

\subsection{Proof of Theorem~\ref{thm:main2}}
We will need:

\begin{Lemma}
\label{lemma:crucial}
Let
$D_{\lambda\mu}^\nu$ be a collection of integers indexed by triples of
shapes such that:

(A) Taking these integers as structure constants on
$\mathbb Z {\mathbb Y}_\Lambda$, by declaring:

$$[\lambda]\star [\mu]=\sum_{\nu} D_{\lambda\mu}^{\nu}[\nu]$$
we obtain a commutative and associative ring. Here $\mathbb Z {\mathbb Y}_{\Lambda}$ denotes the ${\mathbb Z}$-module formally spanned by
all straight shapes $\lambda\subseteq \Lambda$.

(B) $D_{\lambda\rho}^\nu=C_{\lambda\rho}^\nu$ for any $\rho=(t)$.

Then $D_{\lambda\mu}^\nu=C_{\lambda\mu}^\nu$ for all $\lambda,\mu,\nu$.
\end{Lemma}

The above lemma is essentially the same as \cite[Lemma 5.1]{TY5}, which was stated in the setting that
$C_{\lambda,\mu}^{\nu}$ are the Schubert structure constants for the $K$-theory of
(ordinary) Grassmannians. The classes $[{\mathcal O}_{X_{(t)}}]$ can be
shown to
generate $K^{0}(X)$ as an algebra, using the fact that
the corresponding ordinary
cohomology classes $[X_{(t)}]$ generate $H^*(X)$ together with standard facts
about the relationship between $K^0(X)$ and $H^*(X)$.   Thus the same proof
as in \cite{TY5} implies the lemma in this setting also.

Having introduced the notions of $K{\tt rect}$ and $K{\tt infusion}$ in Section~2 (cf. \cite{TY5}), the proof
that (A) holds for the numbers $\{D_{\lambda,\mu}^{\nu}\}$ defined by the {\it jeu de taquin} rule,
is precisely the same as in \cite[Section~5]{TY5}. That part (B) holds
is exactly Corollary~\ref{cor:REUclaim}. Hence the theorem follows. \qed

\section{Structure constants for boundary ideal sheaves}

We first prove the following lemma which is well-known in the context
of classical {\it jeu de taquin}.
An analogous combinatorial fact
was proved by A.~Buch in the setting of $K$-theory of Grassmannians \cite{Buch:KLR}.

Recall that we write $\lambda^\vee$ for the partition obtained
by reflecting $\Lambda/\lambda$ across the antidiagonal or, equivalently,
the partition with distinct parts whose parts are the elements of $\{1,\dots,n\}$
missing from~$\lambda$.

\begin{Lemma}\label{technical} Let $\lambda,\mu$ be shapes in $\Lambda$.
There is a unique increasing filling of $\Lambda/\mu$ which rectifies to
$S_\lambda$ if $\mu=\lambda^\vee$, and otherwise there is none. In other words
$C_{\lambda,\mu}^{\Lambda}=1$ if $\mu=\lambda^{\vee}$ and is zero otherwise.
\end{Lemma}

\begin{proof} The proof is by induction on $n$, the size of the largest part of
$\Lambda$.

Let $T$ be a filling of $\Lambda/\mu$ that rectifies to $S_\lambda$.
Write ${\overline \lambda}=(\lambda_2,\dots)$, which is the shape obtained
by removing the first row from $\lambda$.  Let ${\overline T}$ be the entries of
$T$ which are greater than $\lambda_1$.  By similar reasoning as in the proof of
Theorem~\ref{thm:main}, ${\overline T}$ necessarily rectifies to
$S_{{\overline \lambda}}$.  By induction, ${\overline T}$ must be of shape $\Lambda/{\overline \lambda}^\vee$.

Consider $\widetilde T$ contained in $\widetilde \Lambda$.  By Lemma \ref{lem:double}, $\widetilde T$ rectifies to $\widetilde{S_\lambda}$.
The entries in the bottom
row of $\widetilde T$ form an increasing sequence, so by
Theorem \ref{thm:sixone}, ${\tt
LIS}(\widetilde T)$ is at least the length of this row, whose
length is the same as that of the final column of $\Lambda/\mu$ (which equals
$(\mu^\vee)_1$).  Thus, we have that $\lambda_1 \geq (\mu^\vee)_1$.

At the same time, any ribbon
from the $(\mu^\vee)_1$-th box of the
rightmost column of $\Lambda$ (counting up from the bottom) to a box on
the diagonal of $\Lambda$, has length $(\mu^\vee)_1$.  However,
by construction, the
boxes of
$T\setminus {\overline T}$ form a $\lambda_1$-Pieri filling, and thus include at
least $\lambda_1$ boxes. Since, as already noted,
any Pieri filling must be a ribbon, it follows that $\lambda_1\leq |T\setminus {\overline T}|\leq (\mu^\vee)_1$.

Therefore $\lambda_1=(\mu^\vee)_1$, and hence
$|T\setminus{\overline T}|=\lambda_1$.
It follows that
\[|T|=|T\setminus{\overline T}|+|{\overline T}|=\lambda_1+|{\overline \lambda}|=|\lambda|=|S_\lambda|,\]
and by Poincar\'{e} duality considerations,
the filling $T$ must be the unique one of shape $\Lambda/\lambda^{\vee}$ that witnesses
the cohomological Littlewood-Richardson coefficient $C_{\lambda^{\vee},\lambda}^{\Lambda}=1$.
\end{proof}

We now have the following fact, which was also obtained earlier by \cite{Bu},
in a manner not relying on Lemma~\ref{technical}:

\begin{Lemma}\label{idealform}
\begin{equation}
\label{eqn:pairing}
[\partial X_{\lambda}]=(1-[{\mathcal O}_{(1)}])[{\mathcal O}_{X_{\lambda}}].
\end{equation}
\end{Lemma}

\begin{proof}
To see this, consider the map $\rho:X\to \{pt\}$.
Its pushforward
$\rho_{\star}:K^{0}(X)\to K^{0}(pt)\cong {\mathbb Z}$ allows one to define the pairing
$(\cdot,\cdot):K^{0}(X)\times K^{0}(X)\to {\mathbb Z}$, by $(p,q)=\rho_{\star}(p\cdot q)$.
With respect to
the above pairing,
$\{[{\mathcal O}_{X_{\lambda}}]\}$ and $\{[\partial X_{\lambda^{\vee}}]\}$ are
dual bases
\cite{Bri}.
Hence it suffices to check that the righthand-side of (\ref{eqn:pairing})
defines a dual basis to $\{[{\mathcal O}_{X_{\lambda}}]\}$. Briefly, as in the proof
of the analogous Grassmannian result from \cite[Section 8]{Buch:KLR}, this
follows from Lemma \ref{technical}.
\end{proof}

Given Lemma~\ref{idealform}, the proof of Corollary \ref{cor:main3}
follows the same argument, \emph{mutatis mutandis}, as in \cite[Theorem~1.6]{TY8}.\qed

\section*{Acknowledgements}
We thank Anders Buch for informing us of \cite{Bu} and the work of
Feigenbaum and Sergel.
We also thank Allen Knutson for a helpful communication concerning dual Schubert
bases in $K$-theory. HT is supported by an NSERC Discovery grant.
AY is partially supported by NSF grants DMS-0601010 and DMS-0901331.

\end{document}